\documentclass[12pt]{amsart}

\usepackage{MnSymbol}
\usepackage{mathtools}
\usepackage{csquotes}
\usepackage{stackengine,scalerel}

\newcommand{\N}{\mathbb N}

\newcommand{\Z}{\mathbb Z}
\newcommand{\E}{\mathbb E}

\newtheorem{theorem}{Theorem}[section]
\newtheorem{lemma}[theorem]{Lemma}
\newtheorem{prop}[theorem]{Proposition}
\newtheorem{cor}[theorem]{Corollary}
\newtheorem{corollary}[theorem]{Corollary}

\newtheorem{assumption}{Assumption}

\theoremstyle{definition}

\newtheorem{example}[theorem]{Example}

\theoremstyle{remark}
\newtheorem{remark}[theorem]{Remark}

\newcommand{\rd}{{\mathbb R^d}}
\newcommand{\rr}{{\mathbb R}}

\DeclareMathOperator{\supp}{supp}

\usepackage{comment}
\usepackage{tikz}
\usepackage{float}
\usepackage{relsize}
\usepackage[T1]{fontenc}
\usepackage[english]{babel}
\usepackage{color}
\usepackage{hyperref}
\usepackage{graphicx}

\usepackage[capitalize]{cleveref}

\usepackage{xparse}
\let\realItem\item % save a copy of the original item
\makeatletter
\NewDocumentCommand\myItem{ o }{%
	\IfNoValueTF{#1}%
	{\realItem}% add an item
	{\realItem[#1]\def\@currentlabel{#1}}% add an item and update label
}
\makeatother

\usepackage{enumitem}    
\setlist[enumerate]{
	before=\let\item\myItem,       % use \myItem in enumerate
	label=\textnormal{(\arabic*)}, % format the label
	widest=(2')                    % set the widest label
}
\usepackage{xcolor}
\usepackage[capitalize]{cleveref}
\usepackage{hyperref}

\usepackage{graphicx}
\usepackage{tikz}
\usepackage{atbegshi}

\AtBeginShipoutFirst{%
	\begin{tikzpicture}[remember picture, overlay]
		\node[anchor=north west, xshift=1in, yshift=-1.5cm] at (current page.north west) {%
			\includegraphics[width=2cm]{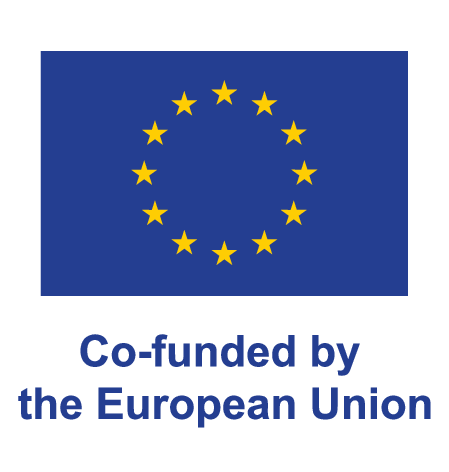} % Change filename, size or position as needed
		};
	\end{tikzpicture}%
}
\usepackage[normalem]{ulem}

\date{\today}

\begin{document}

%	\sloppy

	\title[Lower path regularity]{Lower path regularity in all dimensions} 
	
	\author[M. Hinz]{Michael Hinz$^1$}	
	\address{$^1$ Bielefeld University, Faculty of Mathematics, Postfach 100131, 33501 Bielefeld, Germany}
	\email{mhinz@math.uni-bielefeld.de}
	
	\author[J. M. T\"olle]{Jonas M. T\"olle$^2$}
	\address{$^2$ Aalto University, Department of Mathematics and Systems Analysis, P.O. Box 11100 (Otakaari 1, Espoo), 00076 Aalto, Finland}
	\email{jonas.tolle@aalto.fi}
	
	\author[L. Viitasaari]{Lauri Viitasaari$^3$}
	\address{$^3$ Aalto University, Department of Information and Service Management, P.O. Box 21210 (Ekonominaukio 1, Espoo), 00076 Aalto, Finland}
	\email{lauri.viitasaari@aalto.fi}
	
	\thanks{The research of MH was supported in part by the DFG CRC 1283, \enquote{Taming uncertainty and profiting from randomness and low regularity in analysis, stochastics and their applications}. JMT gratefully acknowledges travel support by the Magnus Ehrnrooth Foundation. JMT was partially supported by the European Union's Horizon Europe research and innovation programme under the Marie Sk\l{}odowska-Curie Actions Staff Exchanges (Grant agreement no.~101183168 -- LiBERA, Call: HORIZON-MSCA-2023-SE-01).}
	
	\begin{abstract}
		We prove precise almost sure lower path regularity results for a wide class of stochastic processes in \emph{all} space dimensions $d\geq 1$. Examples include Gaussian processes, in particular, fractional Brownian motions with Hurst index $H\in (0,1)$, Rosenblatt processes, and solutions to stochastic differential equations driven by fractional Brownian motions with Hurst index $H\in (\frac{1}{4},1)$, all in arbitrary dimensions $d\ge 1$. Our key tool is a new continuity result for Riesz potentials of occupation measures, which we use as substitutes for local times. 
	\end{abstract}
	
	\maketitle
	\tableofcontents
	
	\allowdisplaybreaks

	\section{Introduction}
	
	The almost sure continuity of sample paths is a key research topic in the area of stochastic processes, widely addressed in the literature and with numerous applications. For example, the H\"older continuity of sufficiently high order may be used to define pathwise integrals in the Young \cite{zahle} or rough path sense, see \cite{friz}. A classical result on the almost sure sample path continuity is the famous Kolmogorov-Chentsov criterion, see for instance \cite[Theorem 3.23]{Kallenberg} or \cite[Theorems 1.4.1 and 1.4.4]{Kunita1990}, which allows to translate moment conditions into pathwise upper H\"older regularity. For Gaussian processes, this criterion is known to be also necessary \cite{Azmoodeh-etal}. In the classical case of Brownian motion, more precise path continuity results are provided by L\'evy's modulus of continuity and the law of the iterated logarithm, see for instance \cite[Theorem 1.1.1, resp. Theorem 1.3.3]{CsorgoRevesz1981}. For general Gaussian processes, results on the modulus of continuity were proved in \cite{Dudley1, Dudley2, Fernique, Talagrand}, results related to the law of the iterated logarithm in \cite{Arcones1995, Arcones}. Recently, in \cite{Nummi-Viitasaari}, it was observed that the Kolmogorov-Chentsov criterion is both sufficient and necessary to obtain H\"older continuity results for general hypercontractive processes, extending the results of \cite{Azmoodeh-etal} beyond the Gaussian case. In \cite{Nummi-Viitasaari} the authors  provided a modulus of continuity, exact up to a logarithmic factor, cf. \cite[Corollary 2.11]{Nummi-Viitasaari}. 
	
	To obtain almost sure lower path regularity bounds is a much more challenging task. A prototype result in the classical case of Brownian motion is \cite[Theorem 1.6.1]{CsorgoRevesz1981} by Cs\"org\H{o} and R\'ev\'esz. A well-known way to obtain almost sure lower path regularity is through local times \cite{berman,berman2, GemanHorowitz1980} which are, roughly speaking, defined as the Radon-Nikodym derivatives of the occupation measures with respect to the Lebesgue measure \cite{berman,berman2, GemanHorowitz1980}. Upper regularity results for local times were obtained in \cite{Kesten1965} and \cite{Perkins1981} in the classical case of Brownian motion. For a fairly general class of $\mathbb{R}^d$-valued Gaussian random fields, similar results were shown in \cite[Theorems 1.1 and 1.2]{Xiao1997}. They translate naturally into lower path regularity results for the random field itself \cite[Theorem 3.1]{Xiao1997}. In the special case of the $\mathbb{R}^d$-valued fractional Brownian motion with Hurst index $H\in (0,1)$ the mere existence of local times requires the dimensional condition $Hd<1$ to be satisfied. In \cite{knsv}, the scope of such results was broadened to include also non-Gaussian examples, in particular the Rosenblatt process \cite{Rosenblatt1961}. This process is defined for $H\in \left(\frac12,1\right)$, and since the existence of its local times again needs $Hd<1$, one can only consider the case $d=1$. Inspired by studies of Besov regularity of local times \cite{bouf2,bouf1}, the existence and regularity of local times for a wide class of stochastic processes $X=(X_t)_{t\in [0,T]}$ was investigated further in \cite{SSV}. These results apply to $\mathbb{R}^d$-valued fractional Brownian motions with $Hd<1$, to the Rosenblatt process in dimension $d=1$, and to solutions of stochastic differential equations driven by the $\mathbb{R}^d$-valued fractional Brownian motion with $H\in \left(\frac14,\frac{1}{d}\right)$.

	In the present article, we revisit the class of processes $X=(X_t)_{t\in [0,T]}$ studied in \cite{SSV}, that is, we consider hypercontractive H\"older continuous processes whose increments have characteristic functions satisfying certain scaling and decay properties, see Assumption \ref{a1}.	
We now provide results on almost sure lower path regularity for \emph{all} space dimensions $d\geq 1$. The new key idea is to replace local times by $\alpha$-Riesz potentials of occupation measures. Unlike local times, these potentials exist in $\mathbb{R}^d$ for any $d\geq 1$, so that the severe dimensional restriction incurred with the existence of local times becomes obsolete. Our first main statement is Theorem \ref{main}, in which we prove natural upper regularity results for $\alpha$-Riesz potentials of occupation measures; they parallel \cite[Theorems 1.1 and 1.2]{Xiao1997}. In fact, Theorem \ref{main} is formulated in a way that includes the known local time case, cf. \cite[Theorem 2.1]{SSV}, as the \enquote{limit case} $\alpha=0$, provided that the dimensional condition for the existence of local times is met. In our second main statement, Theorem \ref{cm2}, we give the desired almost sure lower path regularity results for the process $X=(X_t)_{t\in [0,T]}$ itself. They match the known upper path regularity up to possibly different powers of the logarithmic fluctuation terms, see Remark \ref{rem:upper-regularity}. Similarly as before, the choice $\alpha=0$ in the local time case recovers known results, see \cite[p. 147]{Xiao1997} and \cite[Corollary 2.6]{SSV}.
	To illustrate these main statements, we formulate their consequences for Gaussian processes, Rosenblatt processes and solutions to fractional Brownian stochastic differential equations in detail in Theorems \ref{main4}, \ref{main3} and \ref{main2}. In the auxiliary Theorem \ref{thm:holder}, we provide an almost sure H\"older regularity result for Riesz potentials of occupation measures. 
	
	Recently there has been a growing interest in occupation measures and their regularity properties, partially motivated by well-known applications in stochastic analysis, such as regularization by noise \cite{CatellierGubinelli2016, Flandoli2011,GG2022,GG2024,HarangPerkowski2021,RomitoTolomeo2026}
	and existence results for pathwise integrals with discontinuous integrands \cite{Garzon2017,Garzon2023,HTV2022,HTV2023, HTVrough, HTV2025}. The present article continues our series of papers \cite{HTV2022,HTV2023, HTVrough, HTV2025}; they are all linked to the idea of viewing potentials of occupation measures as less precise, \enquote{smeared} variants of local times. This idea is supported by the straightforward observation that local times, if they exist, can be represented as limits of rescaled Riesz potentials of occupation measures, see Lemma \ref{L:abscont} below. A similar kind of result holds even if local times do not exist, but only so-called average densities \cite{BedfordFisher1992,Falconer1997, FalconerXiao1995, Zaehle2001} exist. We take the opportunity to state and prove this observation in the auxiliary Lemma \ref{L:limitsofpots}.
	
	In Section \ref{S:contandlowreg}, we state our main assumptions and main results; applications to specific processes are then discussed in Section \ref{S:apps}. Limits of potentials are considered in Section \ref{S:diff}. Proofs of the main results are provided in Section \ref{S:proofs}. 
	
	We write $\mathcal{L}^d$ for the Lebesgue measure in $\mathbb{R}^d$ and $B(x,r)$ for the open ball with radius $r>0$ and center $x\in\mathbb{R}^d$.
	The volume of the unit ball in $\mathbb{R}^d$ is denoted by $\omega_d$, the Euclidean norm on $\mathbb{R}^d$ by $\|\cdot \|$.
	
	\section{Continuity of potentials and lower path regularity}\label{S:contandlowreg}
	
	Let $T>0$ be fixed throughout the entire article. We consider an $\mathbb{R}^d$-valued stochastic process $X=(X_t)_{t\in[0,T]}$ on a probability space $(\Omega,\mathcal{F},\mathbb{P})$. Given $0\leq s<t\leq T$, we write $\mu_{s,t}^X$ for the \emph{occupation measure} of $X$ within the interval $[s,t]$; it is defined as the random measure
	\[\mu_{s,t}^X(A):=\mathcal{L}^1(\{u\in [s,t]:\ X_u\in A\}),\qquad \text{$A\subset\mathbb{R}^d$ Borel.}\]
	For any fixed Borel set $A\subset\mathbb{R}^d$ and $\omega\in \Omega$, the map 
	\[[s,t]\to \mu_{s,t}^{X(\omega)}(A)=\int_s^t\mathbf{1}_A(X_u(\omega))du\] 
	defines a finite Borel measure on $[0,T]$, which is absolutely continuous with respect to $\mathcal{L}^1$. Clearly $\mu_{s,t}^X=\mu_{0,t}^X-\mu_{0,s}^X$.
	
	We say that \emph{$X$ has continuous local times} if there is a random field 
	\[\{L_t^X(x): t\in [0,T], x\in \mathbb{R}^d\}\] 
	such that $(t,x)\mapsto L_t^X(x)$ is $\mathbb{P}$-a.s. continuous on $[0,T]\times \mathbb{R}^d$ and for any $t\in [0,T]$ we have 
	\[\mu_{0,t}^X(dx)=L_t^X(x)\mathcal{L}^d(dx)\qquad \text{$\mathbb{P}$-a.s.}\]

	For any $0<\alpha<d$ the $\alpha$-Riesz potential of $\mu_{s,t}^X$ is the $[0,+\infty]$-valued function 
	\begin{equation}\label{E:occupot}
		U^\alpha\mu_{s,t}^X(x):=\int_{\mathbb{R}^d}\|x-y\|^{\alpha-d}\mu_{s,t}^X(dy),\qquad x\in\mathbb{R}^d.
	\end{equation}
	For convenience, we omit the customary multiplicative constant.
	For any fixed $x\in\mathbb{R}^d$, the map
	\begin{equation}\label{E:measuresonline}
		[s,t]\mapsto U^\alpha\mu_{s,t}^X(x)=\int_s^t \|x-X_u\|^{\alpha-d}du
	\end{equation}
	defines a Borel measure on $[0,T]$, which is absolutely continuous with respect to $\mathcal{L}^1$, although not necessarily finite.
	
	If $X$ has continuous local times, then, roughly speaking, $L_{t}^X$ can be regarded as the limit of $\frac{\alpha}{d\omega_d} U^\alpha\mu_{0,t}^X$
	as $\alpha\to 0$; a rigorous statement is provided in Lemma \ref{L:abscont}. Motivated by this observation, we set, for all $0\leq s<t\leq T$, $x\in\mathbb{R}^d$, 
	\[L^{\alpha,X}(x, [s,t]):=\begin{cases} \frac{\alpha}{d\omega_d} U^\alpha\mu_{s,t}^X(x)\ &\text{for $0<\alpha<d$},\\ 
		L_t^X(x)-L_s^X(x)\ &\text{for $\alpha=0$.}\end{cases}\]
	
	Our main result is Theorem \ref{main} below. It provides upper density bounds for the measures \eqref{E:measuresonline}, which can also be seen as continuity in time properties of the potentials in \eqref{E:occupot}. We state Theorem \ref{main} under the following assumption.
	
	\begin{assumption}\label{a1}	 Let $H\in (0,1)$.
		\begin{enumerate}%[label={(\roman*)},itemindent=1em]
			\item[(i)] \emph{(decay and scaling of the characteristic function)}
			There exist $c_0>0$ and $\theta\geq 0$ such that for any $n\in \N$ and $\xi_j = (\xi_j^1, \ldots, \xi_j^d) \in (\rr\setminus \{0\} )^d$, $j=1, \ldots, n$,  any partition  $0=t_0  <t_1 < \ldots < t_n<T$, and any $k_{j,\ell} \in \{0,4\}$ the inequality
			\begin{align*}
				\Big\lvert  \mathbb{E} \Big[ \exp \Big( i \sum_{j=1}^n \langle \xi_j, X_{t_{j}}-X_{t_{j-1}}\rangle \Big) \Big]\Big\rvert  \leq c_0^n n^{H\theta n}\prod_{j=1}^n \prod_{\ell=1}^d\frac{1}{|\xi_j^\ell|^{k_{j,\ell}} 
					(t_j-t_{j-1})^{H k_{j,\ell}}}
			\end{align*}
			holds.
			\item[(ii)]  \emph{(hypercontractivity)} There exist $c_1\geq 0$ and $\iota\geq 0$ such that for any $p\geq 1$ and all $0\leq s<t\leq T$ we have
			$$
			\mathbb{E}\Vert X_t-X_s\Vert^p \leq c_1^p p^{\iota p}|t-s|^{Hp}.
			$$
		\end{enumerate}
	\end{assumption}	
	
	This assumption is satisfied for various classes of processes which in general do not have local times, see Section \ref{S:apps} below.

	\begin{theorem} \label{main}
		Suppose that $X=(X_t)_{t\in[0,T]}$ satisfies Assumption \ref{a1} and that
		\begin{equation}\label{E:parameters}
			0\leq \alpha<d\quad\text{if $Hd<1$}\qquad\text{respectively}\qquad \max\big(0,d-\frac{1}{H}\big)<\alpha<d\quad\text{if $Hd\geq 1$}.
		\end{equation}
		
		\begin{enumerate}
			\item[(i)] For any $t \in (0,T)$  there is a constant $c_2(t)>0$ such that
			\begin{equation}
				\label{eq:main-fixed-s} 
				\limsup _{r \to 0}\: \sup_{x\in\rr^d}\:\frac{L^{\alpha,X}(x,[t-r,t+r])}{r^{1-H(d-\alpha)} (\log \log r^{-1})^{H(\theta+d-\alpha)}} \leq c_2(t)\qquad \text{$\mathbb{P}$-a.s.}
			\end{equation}
			\item[(ii)] There is a constant $c_3>0$ such that
			\begin{equation}
				\label{eq:main-sup-s} 
				\limsup _{r \to 0} \sup_{t\in(0,T)} \sup_{x\in\rr^d} \frac{L^{\alpha,X}(x,[t-r,t+r])}{r^{1-H(d-\alpha)} (\log r^{-1})^{H(\theta+d-\alpha)}} \leq c_3 \qquad \text{$\mathbb{P}$-a.s.}
			\end{equation}
		\end{enumerate}
	\end{theorem}

	\begin{remark} For $\alpha=0$ condition \eqref{E:parameters} requires $dH<1$. In this case the process $X$ has continuous local times and Theorem \ref{main} reproduces the corresponding results stated in \cite[Theorem 2.1]{SSV}. 
	\end{remark}
	
	Proceeding similarly as in \cite{knsv} we obtain the following new result on lower path regularity, which is our second main result.

	\begin{theorem} \label{cm2}
		Suppose that $X=(X_t)_{t\in[0,T]}$ satisfies Assumption \ref{a1} and that $\alpha$ is as in \eqref{E:parameters}. 
		\begin{enumerate}
			\item[(i)] For any $t \in (0,T)$  there is a constant $c_4(t)>0$ such that
			\begin{equation}\label{E:output}		
				\liminf _{r \to 0} \sup_{s \in (t-r,t+r)} \frac{\|X_t - X_s \| }{r^H(\log\log r^{-1})^{-H(\frac{\theta}{d-\alpha}+1)}} \geq c_4(t)\qquad \text{$\mathbb{P}$-a.s.}
			\end{equation}
			\item[(ii)] There is a constant $c_5>0$ such that		
			\begin{equation}\label{E:output2}		
				\liminf _{r \to 0} \inf_{t\in (0,T)} \sup_{s \in (t-r,t+r)} \frac{\|X_t - X_s \| }{r^{H} ( \log r^{-1})^{-H(\frac{\theta}{d-\alpha}+1)}} \geq c_5\qquad \text{$\mathbb{P}$-a.s.}
			\end{equation}
		\end{enumerate}
		In particular, $X$ is almost surely nowhere differentiable.
	\end{theorem}
	
	\begin{remark} For $\alpha=0$ (and therefore $Hd< 1$) Theorem \ref{cm2} recovers the corresponding results in \cite[Corollary 2.6]{SSV}\footnote{In \cite[Corollary 2.6]{SSV} one has $r^{Hd}$ in the denominator instead of $r^{H}$. However, by carefully examining the proof, e.g. in \cite{knsv}, one observes that the correct term is $r^{H}$, as expected.}. For $Hd\geq 1$ the parameter $\alpha$ can be chosen arbitrarily close to $d-\frac{1}{H}$.
	\end{remark}		
	
	\begin{remark}
		\label{rem:upper-regularity}
		Under Assumption \ref{a1} (ii) \cite[Corollary 2.11]{Nummi-Viitasaari} gives 
		$$
		\limsup_{|t-s|\to 0} \frac{\Vert X_t - X_s\Vert}{|t-s|^H \left(\log\frac{1}{|t-s|}\right)^\iota} \leq c_6
		$$
		with a constant $c_6>0$. Since $r\mapsto r^H \left(\log r^{-1}\right)^\iota$ is increasing for all small enough $r>0$, this implies that
		\begin{equation}\label{E:match}
			\limsup _{r \to 0} \sup_{t\in(0,T)}\sup_{s \in (t-r,t+r)} \frac{\|X_t - X_s \| }{r^H(\log r^{-1})^{\iota}} \leq c_6.
		\end{equation}
		Up to different exponents in the logarithmic terms, \eqref{E:output2} and \eqref{E:match} provide matching lower and upper regularity results.
	\end{remark}
	
	We finally make the simple but useful observation that, in contrast to the situation in \cite[Theorem 2.1]{SSV}, Assumption \ref{a1} does not include any dimensional restriction on $H\in (0,1)$. As a consequence, it is stable under forming random vectors from independent components.
	
	\begin{lemma}
		\label{lma:1-to-d}
		Suppose that $X=(X^1,\dots,X^d)$ is a vector of independent real valued processes $X^\ell$ each satisfying Assumption \ref{a1} with $1$ in place of $d$ and  with $H$, $c_0$, $\theta$, $c_1$ and $\iota$. Then $X$ satisfies Assumption \ref{a1} with $H$, $c_0^d$, $d\theta$, $dc_1$ and $\iota$.
	\end{lemma}
	\begin{proof}
		Since each component $X^\ell$ satisfies 
		\begin{align*}
			\Big\lvert\mathbb{E} \Big[ \exp \Big( i \sum_{j=1}^n (\xi_j^\ell( X^\ell_{t_{j}}-X^\ell_{t_{j-1}})) \Big) \Big]\Big\rvert \leq c_0^n n^{H\theta n}\prod_{j=1}^n \frac{1}{\lvert\xi_j^\ell\rvert^{k_{j,\ell}} 
				(t_j-t_{j-1})^{H k_{j,\ell}}},
		\end{align*}
		independence gives
		\begin{align*}
			\Big\lvert  \mathbb{E} \Big[ \exp \Big( i \sum_{j=1}^n \langle \xi_j, X_{t_{j}}-X_{t_{j-1}}\rangle \Big) \Big]\Big\rvert  &= 	  \prod_{\ell=1}^d \Big\lvert\mathbb{E} \Big[ \exp \Big( i \sum_{j=1}^n \xi_j^\ell( X^\ell_{t_{j}}-X^\ell_{t_{j-1}}) \Big) \Big]\Big\rvert \\
			&\leq		
			c_0^{dn} n^{dH\theta n}\prod_{j=1}^n \prod_{\ell=1}^d\frac{1}{\lvert\xi_j^\ell \rvert^{k_{j,\ell}} 
				(t_j-t_{j-1})^{H k_{j,\ell}}}
		\end{align*}
		Moreover, we have 
		\[
		\mathbb{E}\Vert X_t -X_s\Vert^p \leq\E\left(\sum_{\ell=1}^d |X_t^\ell-X_s^\ell|\right)^p\leq d^{p-1}\sum_{\ell=1}^d  \mathbb{E}|X_t^\ell-X_s^\ell|^p\leq d^p c_1^p p^{\iota p}|t-s|^{Hp}.
		\]
	\end{proof}

	\section{Gaussian processes, Rosenblatt processes and fractional Brownian SDEs}\label{S:apps}
	
	We consider three classes of examples, namely Gaussian processes, Rosenblatt processes and solutions of stochastic differential equations driven by fractional Brownian motions. In each of these cases, Theorem \ref{cm2} gives new results on lower path regularity.
	
	\subsection*{Gaussian processes} 
	For Gaussian random fields in the local time regime, results on lower path regularity were provided in \cite[Theorem 3.1]{Xiao1997}. For the Gaussian process case ($N=1$ in \cite{Xiao1997}) we provide a corresponding result beyond the local time regime.
	
	\begin{theorem} \label{main4}
		Let $X=(X_t)_{t\in[0,T]}$ be a centered Gaussian process $X=(X^1,...,X^d)$ with values in $\mathbb{R}^d$ and with real valued components $X^\ell=(X^\ell_t)_{t\in[0,T]}$ satisfying
		\begin{equation}
			\label{eq:Gaussian-variance}
			C_- \lvert t-s\rvert^{2H} \leq \mathbb{E}(X^\ell_t - X^\ell_s)^2 \leq C_+\lvert t-s\rvert^{2H},\qquad 0\leq s<t\leq T,
		\end{equation}
		for fixed constants $H \in (0,1)$ and $C_+\geq C_->0$. Suppose that the local non-determinism property
		\begin{equation}
			\label{eq:Gaussian-LND}
			\E \left| \sum_{k=1}^m \langle \xi_k,X_{t_k}-X_{t_{k-1}}\rangle\right|^2 \geq C \prod_{\ell=1}^d \sum_{k=1}^m \xi_{k,\ell}^2 \E (X^\ell_{t_k}-X^\ell_{t_{k-1}})^2
		\end{equation}
		holds with a constant $C>0$ and for all $m\in \mathbb{N}$, $\xi_k \in \mathbb{R}^d$ and all $0=t_0<t_1<\ldots<t_m\leq T$. Then 
		\begin{enumerate}
			\item[(i)] Assumption \ref{a1} (i) and (ii) hold with $\theta=0$ respectively $\iota = \frac12$.
			\item[(ii)] For any $t \in (0,T)$  there is a constant $c_4(t)>0$ such that
			\[\liminf _{r \to 0} \sup_{s \in (t-r,t+r)} \frac{\|X_t - X_s \| }{r^H(\log\log r^{-1})^{-H}} \geq c_4(t)\qquad \text{$\mathbb{P}$-a.s.}\]
			\item[(iii)] There is a constant $c_5>0$ such that		
			\[\liminf _{r \to 0} \inf_{t\in (0,T)} \sup_{s \in (t-r,t+r)} \frac{\|X_t - X_s \| }{r^{H} ( \log r^{-1})^{-H}} \geq c_5\qquad \text{$\mathbb{P}$-a.s.}\]
		\end{enumerate}				
	\end{theorem}
	
	\begin{remark}
		The resulting constants $c_0$ and $c_1$ in Assumption \ref{a1} (i) respectively (ii) depend only on  $T$, $d$, $C_-$ and $C$ respectively $T$, $d$ and $C_+$.
	\end{remark}
	
	\begin{proof}
		Item (i) was already shown in the proof of \cite[Proposition 2.12]{SSV}; the restriction $Hd<1$ required there is not needed here. Items (ii) and (iii) follow from (i) and Theorem \ref{cm2}.
	\end{proof}

	\begin{remark}
		Theorem \ref{main4} gives lower path regularity for an \emph{arbitrary} pair $(H,d)$. For $Hd<1$ it recovers the corresponding result provided in \cite{Xiao1997}, for $Hd\geq 1$ the result is new. 
	\end{remark}
	
	\begin{example}
		Theorem \ref{main4} covers the case of $d$-dimensional fractional Brownian motion $B=(B^1, \ldots, B^d)$ with \emph{arbitrary} Hurst index $H\in (0,1)$; it is defined by taking independent centered Gaussian processes $B^\ell = (B^\ell_t)_{t\in [0,T]}$, $1 \leq \ell \leq d$, each having covariance function \[R(s,t) = \frac{1}{2}\left[t^{2H}+s^{2H}-|t-s|^{2H}\right],\quad  0\leq s<t\leq T,\] 
		and forming $B$ as stated. The special case for $H=\frac12$ is Brownian motion. If $Hd<1$, then $B$ has continuous local times $L_T^{B}$. In this case Theorem \ref{main} recovers the $N=1$ case of \cite[Corollary 1.1]{Xiao1997}.
	\end{example}

	\subsection*{Rosenblatt processes}
	Let $Z=(Z_t)_{t\in [0,T]}$ be a $d$-dimensional Rosenblatt process, that is, a process $Z=(Z^1,\dots, Z^d)$ whose components $Z^\ell$ are independent real valued Rosenblatt processes of Hurst index $H\in \left(\frac12,1\right)$ which can be constructed via the representation 
	$$Z_t^\ell = \int_{\rr^2}\mathbf{1}_{\{x \neq \pm y\}}(x,y) \frac{e^{it (x+y)} -1}{i(x+y)} Z_G(dx) Z_G(dy)$$
	for every $t \in [0,T]$, where $Z_G(dx)$ is a complex-valued random white noise with control measure $G$ satisfying $G(dx) = \lvert x\rvert^{-H} dx$.

	Rosenblatt processes are non-Gaussian processes having correlation structures similar to those of fractional Brownian motion. They arise naturally as scaling limits of long-memory sequences. We refer to \cite{knsv} for a more detailed introduction and discussion. 
	
	Lower path regularity results for real valued Rosenblatt processes were provided in \cite[Corollary 1.7]{knsv} and in \cite[Corollary 2.6 and Proposition 2.10]{SSV}. The following result is, to the best of our knowledge, the first extension to arbitrary dimensions $d$. 
	
	\begin{theorem} \label{main3}
		Let $Z=(Z_t)_{t\in[0,T]}$ be a $d$-dimensional Rosenblatt process $Z=(Z^1,\dots, Z^d)$ with Hurst index $H\in (\frac12, 1)$. 
		\begin{enumerate}
			\item[(i)] Assumption \ref{a1} (i) and (ii) hold with $\theta=0$ respectively $\iota = 1$.
			\item[(ii)] Statements (ii) and (iii) in Theorem \ref{main4} hold with $Z$ in place of $X$.
		\end{enumerate}		
	\end{theorem}

	\begin{proof}	
		Item (i) follows in a straightforward manner by combining Lemma \ref{lma:1-to-d} and \cite[Proof of Proposition 2.10]{SSV}. Items (ii) and (iii) are again due to (i) and Theorem \ref{cm2}.
	\end{proof}
	
	\subsection*{Solutions to SDEs driven by fractional Brownian motion} Let $B$ be a $d$-dimensional fractional Brownian motion with Hurst parameter $H \in (\frac{1}{4}, 1)$. We consider the class of differential equations given by
	\begin{equation}\label{SDE1}
		X_t = x + \int_0^t V_0 (X_s) ds + \sum_{\ell=1}^d \int_0^t V_\ell(X_s) dB_s^\ell,\qquad t \in [0,T],
	\end{equation}
	where $x \in \rd$ is the initial condition and $V_0, V_1, \ldots, V_d$ are given vector fields in $\rd$. For the Brownian case $H=\frac12$ the second integral is understood in the Stratonovich sense, for  $H\in (\frac12,1)$ in the pathwise Young sense \cite{zahle} and for $H\in (\frac14,\frac12)$ in the rough path sense \cite{friz}. The lower bound $H>\frac14$ is imposed in order to ensure the existence of the stochastic integrals. The vector fields $V_0, V_1, \dots, V_d$ are assumed to be smooth with bounded derivatives of all orders and such that the $(d\times d)$-matrix valued function $V=[V_1,\dots,V_d]$ with columns $V_1,\dots,V_d$ is uniformly elliptic, that is, 
	\begin{equation}\label{E:assvector}
		V_0, V_1, \ldots, V_d \in \mathcal{C}_b^\infty (\rd,\rd)\qquad\text{and}\qquad \xi^{\operatorname{T}} V(x) V(x)^{\operatorname{T}} \xi \geq \lambda \| \xi\| ^2 \qquad x,\xi \in \rd,
	\end{equation}
	with a constant $\lambda>0$. Under these standard assumptions the existence and uniqueness of solutions $X=(X_t)_{t\in[0,T]}$ in the respective sense can be guaranteed \cite{baudoin,lou}.
	
	The existence of continuous local times $L_T^X$ for $X$ requires $Hd<1$, see \cite[Theorem 1.1]{lou} and \cite[Theorem 2.7]{SSV}. Results on the lower path regularity of $X$ under this restriction were proved in \cite[Corollary 2.6 and Theorem 2.7]{SSV}; they require  $H\in \left(\frac14,\frac12\right)$ for 
	$d=2$, $H\in \left(\frac14,\frac13\right)$ for $d=3$, while for $d\geq 4$ no results were available. 
	
	Using Riesz potentials of occupation measures, we can now provide a result for the entire range $H\in (\frac{1}{4}, 1)$. This highlights the benefits of our approach.
	
	\begin{theorem} \label{main2} Assume \eqref{E:assvector}, let $B$ be a $d$-dimensional fractional Brownian motion with Hurst parameter $H \in (\frac{1}{4}, 1)$ and let $X=(X_t)_{t\in[0,T]}$ be the solution $X=(X^1,\dots, X^d)$ of \eqref{SDE1}, understood as stated. 
		\begin{enumerate}
			\item[(i)]  For every $\delta>0$, Assumption \ref{a1} (i) holds with $\theta=\frac{8d+\delta}{H}$, and Assumption \ref{a1} (ii) holds
			with $\iota=\frac12$.
			\item[(ii)] For any $\gamma>8\max(1,Hd)+H$ and any $t \in (0,T)$  there is a constant $c_4(t)>0$ such that
			\[\liminf _{r \to 0} \sup_{s \in (t-r,t+r)} \frac{\|X_t - X_s \| }{r^H(\log\log r^{-1})^{-\gamma}} \geq c_4(t)\qquad \text{$\mathbb{P}$-a.s.}\]
			\item[(iii)] For any $\gamma>8\max(1,Hd)+H$ there is a constant $c_5>0$ such that		
			\[\liminf _{r \to 0} \inf_{t\in (0,T)} \sup_{s \in (t-r,t+r)} \frac{\|X_t - X_s \| }{r^{H} ( \log r^{-1})^{-\gamma}} \geq c_5\qquad \text{$\mathbb{P}$-a.s.}\]
		\end{enumerate}
	\end{theorem}
	
	\begin{remark}
		For $dH<1$ we recover \cite[Theorem 2.7]{SSV}, for $dH\geq 1$ the result is new.
	\end{remark}
	
	\begin{proof}
		The proof of \cite[Theorem 2.7]{SSV} shows Assumption \ref{a1} as stated in item (i), the relevant arguments there do not use the condition $Hd<1$. Items (ii) and (iii) are again consequences of (i) and Theorem \ref{cm2}. 
	\end{proof}

	\section{Limits of rescaled potentials}\label{S:diff}
	
	We comment on a deterministic averaging argument supporting the point of view that potentials of occupation measures are less precise substitutes for local times. Although already used \cite{Hinz2005} and quoted in the literature \cite{CalefHardin2009, FanLaugesen2025}, we have not seen a full statement and proof published anywhere. We provide a new and refined variant, suitable for applications to transient stochastic processes.
	
	We write $\mu\ll \mathcal{L}^d$ to say that a measure $\mu$ on $\mathbb{R}^d$ is absolutely continuous with respect to the $d$-dimensional Lebesgue measure $\mathcal{L}^d$. The Radon-Nikodym density $\frac{d\mu}{d\mathcal{L}^d}$ of such a measure is the limit of rescaled Riesz-potentials.

	\begin{lemma}\label{L:abscont}
		If $\mu$ is a finite Borel measure on $\mathbb{R}^{d}$ such that $\mu\ll \mathcal{L}^d$, then 
		\begin{equation}\label{E:differentiate}
			\lim_{\alpha \to 0} \alpha\;U^{\alpha}\mu = d\:\omega_d\;\frac{d\mu}{d\mathcal{L}^d}\qquad \text{$\mu$-a.e.}
		\end{equation}
		If $L$ is a version of $\frac{d\mu}{d\mathcal{L}^d}$ which is bounded on $\mathbb{R}^d$ and 
		continuous at a point $x\in\mathbb{R}^d$, then $\lim_{\alpha \to 0} \alpha\;U^{\alpha}\mu(x)=d\:\omega_d\;L(x)$. 
	\end{lemma}
	
	\begin{proof}
		Writing $m_x(r):=\mu(B(x,r))$ and integrating by parts,
		\begin{equation}\label{E:firstblock0}
			\int_{\mathbb{R}^d}|x-y|^{\alpha-d}\mu(dy)=\int_0^\infty r^{\alpha-d}m_x(dr)=(d-\alpha)\int_0^\infty m_x(r)r^{\alpha-d-1}dr
		\end{equation}
		for any $0<\alpha<d$ at all $x\in\mathbb{R}^d$ outside a $\mu$-null set. We have used that $\lim_{R\to\infty}m_x(R)R^{\alpha-d}=0$, which follows since $\mu$ is finite, and $\lim_{r\to 0}m_x(r)r^{\alpha-d}=0$ $\mu$-a.e. on $\mathbb{R}^d$ with the null set not depending on $\alpha$, which follows using Lebesgue's differentiation theorem, \cite[Corollary 2.14]{Mattila1995}. Using this theorem together with the properties of the averaging kernels $r\mapsto \mathbf{1}_{(0,1)}(r)\alpha r^{\alpha-1}$, $\alpha>0$, it is easily seen that 
		\[\lim_{\alpha\to 0}\alpha (d-\alpha)\int_0^1 \frac{m_x(r)}{\omega_d r^d}r^{\alpha-1}dr=d\frac{d\mu}{d\mathcal{L}^d}(x)\]
		at $\mu$-a.e. $x\in\mathbb{R}^d$.
		Since the finiteness of $\mu$ also implies that
		\begin{equation}\label{E:large0}
			\lim_{\alpha\to 0}\alpha(d-\alpha)\int_1^\infty m_x(r)r^{\alpha-d-1}dr=0,
		\end{equation}
		the first statement follows. The second statement follows similarly.
	\end{proof}
	
	\begin{remark}
		If $\mu\ll \mathcal{L}^d$, then the Radon-Nikodym density $\frac{d\mu}{d\mathcal{L}^d}\in L^1(\mathbb{R}^d)$ equals the convolution of itself with a Dirac probability measure $\delta_0$ at zero. By Lemma \ref{L:abscont} the probability measures
		$\mathbf{1}_{(0,1)}(|y|)\frac{\alpha}{d\omega_d}|y|^{\alpha-d}dy$, $0<\alpha<d$, converge weakly to $\delta_0$ as $\alpha\to 0$.
	\end{remark}
	
	If a stochastic process has continuous local times, then \eqref{E:differentiate} recovers them. 
	
	\begin{example}\mbox{}\label{Ex:classical0} 
		\begin{enumerate}
			\item[(i)] For $d=1$ and any $1<\beta\leq 2$ the symmetric $\beta$-stable process $X$ on $\mathbb{R}$ has spatially continuous local times $L_T^X$ on $[0,T]$, as shown in \cite{Boylan1964}. As a consequence, there is an event of probability one on which  
			$\lim_{\alpha\to 0}L^{\alpha,X}(x,[0,T])=L_T^X(x)$, $x\in \mathbb{R}$, as can be seen using Lemma \ref{L:abscont} and a cut-off argument. In the special case $\beta=2$ the process $X$ is Brownian motion and $L_T^X$ is the Brownian local time \cite{Levy1948, Trotter1958}.
			\item[(ii)] If $B$ is a fractional Brownian motion on $\mathbb{R}^d$ with Hurst index $H$ such that $Hd<1$, then $\lim_{\alpha\to 0}L^{\alpha,B}(x,[0,T])=L_T^{B}(x)$, $x\in \mathbb{R}$, on an event of probability one \cite{berman,Pitt1978}.
		\end{enumerate}
	\end{example}

	Even if $\mu$ is not absolutely continuous with respect to $\mathcal{L}^d$, one can, in some cases, observe an analog of \eqref{E:differentiate}. Let $\mu$ be a finite Borel measure on $\mathbb{R}^{d}$, $0<s\leq d$ and $x\in \mathbb{R}^d$. If the limit
	\begin{equation}\label{E:ad}
		D^{s}_{a}\mu(x):=\lim_{u \to 0}\frac{1}{\left|\log u \right|}\int_{u}^{1}
		\frac{\mu(B(x,r))}{r^{s}}\frac{1}{r}\:dr 
	\end{equation}
	exists, it is called \emph{average $s$-density} of $\mu$ at $x$. The right-hand side in \eqref{E:ad} can alternatively be written
	$\lim_{T \to \infty}\frac{1}{T}\int_{0}^{T}\mu\left(B\left(x,e^{-t}\right)\right)e^{ts}dt$. 
	
	\begin{remark}\mbox{}
		\begin{enumerate}
			\item[(i)] The definition \eqref{E:ad} is an averaged differentiation of measures. To the $\lim_{r\to 0}$ of the ratio $\frac{\mu(B(\cdot,r))}{r^{s}}$, if it exists, one refers as the \emph{$s$-density} of $\mu$ at $x$. If it exists, then also $D^{s}_{a}\mu(x)$ exists and the two agree, as can easily be seen from the properties of the averaging kernels $r\mapsto \mathbf{1}_{(u,1)}(r)|\log u|^{-1}r^{-1}$, $u>0$. In particular, if 
			$s=d$ and $\mu\ll \mathcal{L}^d$, then by Lebesgue's differentiation theorem the $d$-density of $\mu$ exists $\mathcal{L}^d$-a.e. on $\mathbb{R}^d$ and equals $\omega_d\:\frac{d\mu}{d\mathcal{L}^d}$. Accordingly, $D^{d}_{a}\mu=\omega_d\:\frac{d\mu}{d\mathcal{L}^d}$ $\mu$-a.e.
			If $L$ is a version of $\frac{d\mu}{d\mathcal{L}^d}$ which is bounded on $\mathbb{R}^d$ and continuous at $x\in\mathbb{R}^d$, then $D^{d}_{a}\mu(x)=\omega_d\:L(x)$.
			\item[(ii)] For points $x\in \mathbb{R}^d\setminus (\supp\mu)$ the $s$-density of $\mu$ at $x$ exists and equals zero. As a consequence, $D^{s}_{a}\mu(x)=0$ at such $x$.
		\end{enumerate}
	\end{remark}
	
	\begin{remark}\label{R:literature}\mbox{}
		\begin{enumerate}
			\item[(i)]  Average densities should be seen in connection with the prominent results in \cite{Besicovich1928, Marstrand1964, Preiss1987}, see \cite{Falconer1995, Mattila1995, Moerters1997}. They were introduced in \cite{BedfordFisher1992} for situations where $s$-densities do not exist. Their $\mu$-a.e. existence and constancy was proved in \cite{Falconer1992, Graf1995, PatschkeZaehle1993} and \cite{Zaehle2001} for the cases of self-similar respectively self-conformal measures $\mu$; results for random fractals are provided in \cite{PatschkeZaehle1994}. See also \cite[Section 6.2]{Falconer1997}.
			\item[(ii)] Let $d\geq 1$, $0<\beta\leq 2$, $\beta<d$, $\beta\neq 1$ and let $X=X(\omega,\cdot)$ be the isotropic $\beta$-stable process on $\mathbb{R}^d$. 
			In \cite[Theorem 1]{FalconerXiao1995} it was proved that there is an event $\Omega_0$ of probability one and a constant $c_\beta>0$, such that for any $\omega\in \Omega_0$ the average $\beta$-density $D_a^\beta\mu_\omega$ of the occupation measure $\mu_{0,T}^{X(\omega,\cdot)}$ exists and equals $c_\beta$ at all $x\in \{X(\omega,t):\ t\in [0,T]\}\setminus N_\omega$, where $N_\omega$ is a $\mu_{0,T}^{X(\omega,\cdot)}$-null set. For $d\geq 3$ and $\beta=2$ the process $X$ is Brownian motion $B=B(\omega,\cdot)$ on $\mathbb{R}^d$. Related results for Brownian motion in $\mathbb{R}^2$ were provided in \cite{Moerters1998}.
		\end{enumerate} 
	\end{remark}
	
	In \cite{Graf1995, PatschkeZaehle1994} and in \cite[Section 4.2]{Zaehle2001}, it was observed that average $s$-densities are limits of truncated Riesz potentials of order $d-s$. The arguments in \cite{Zaehle2001}, although stated under more restrictive assumptions, give the following.
	
	\begin{lemma}
		Let $\mu$ be a finite Borel measure on $\mathbb{R}^{d}$, $0<s\leq d$, and let $x\in\mathbb{R}^d$ be such that 
		$\limsup_{r\to 0}\frac{\mu(B(x,r))}{r^s|\log r|}=0$.
		Then $D_a^s\mu(x)$ exists if and only if 
		$\lim_{r\to 0}\frac{1}{s|\log r|}\int_{\mathbb{R}^d\setminus B(x,r)}|x-y|^{-s}\mu(dy)$
		exists, and in this case the two values agree.
	\end{lemma}
	
	Our next observation is an  approximation of $D_a^s\mu(x)$ in which the order of the Riesz potential varies and which generalizes \eqref{E:differentiate}. Recall that a positive Borel function $\ell$ on $(0,\varepsilon)$, where $\varepsilon>0$, is said to be slowly varying at $0$ if $\lim_{h\to 0}\frac{\ell(\lambda h)}{\ell(h)}=1$ for all $\lambda>0$.
	\begin{lemma}\label{L:limitsofpots}
		Let $\mu$ be a finite Borel measure on $\mathbb{R}^{d}$, $0<s\leq d$, and let $x\in\mathbb{R}^d$ be such that 
		\begin{equation}\label{E:upperboundmu}
			\limsup_{r\to 0}\frac{\mu(B(x,r))}{r^{s}\ell(r)}<\infty,
		\end{equation}
		where $\ell$ is a positive Borel function slowly varying at zero. If $D_{a}^{s}\mu(x)$ exists, then
		\begin{equation}\label{E:limitofpot}
			\lim_{\varepsilon \to 0} \varepsilon\;U^{d-s+\varepsilon}\mu(x) = s\;D_{a}^{s}\mu(x).
		\end{equation}
	\end{lemma}
	
	Consider the averaging kernels $u\mapsto k_{\varepsilon}(u):=\mathbf{1}_{(0,1)}(u)\varepsilon^2 u^{\varepsilon-1}|\log u|$, $\varepsilon>0$. 
	It is easily seen that for any $u_0>0$ we have 
	\begin{equation}\label{E:basicprops}
		\int_0^\infty k_\varepsilon(u)\:du=1,\qquad \lim_{\varepsilon \to 0}\int_{u_0}^{\infty}k_{\varepsilon}(u)\:du = 0\qquad\text{and}\qquad \lim_{\varepsilon \to 0}\int_{0}^{u_0}k_{\varepsilon}(u)\:du = 1.
	\end{equation}
	If $f: [0,+\infty) \longrightarrow [0,+\infty)\ $ is a bounded function, continuous at some $z \geq 0$, then, by \eqref{E:basicprops},
	\begin{equation}\label{E:basicapprox}
		\lim_{\varepsilon \to 0}\int_{0}^{\infty}f(u)k_{\varepsilon}(z+u)\:du = f(z).
	\end{equation}
	For any $\varepsilon>0$ the kernel $k_\varepsilon$ satisfies
	\begin{equation}\label{E:specific}
		\int_0^r\frac{k_\varepsilon(u)}{|\log u|}\:du=\varepsilon\:r^\varepsilon,\qquad 0\leq r<1.
	\end{equation}
	Combining these facts, one can give a quick proof of Lemma \ref{L:limitsofpots}.
	
	\begin{proof}[Proof of Lemma \ref{L:limitsofpots}.]
		As in \eqref{E:firstblock0}, we find that
		\begin{equation}\label{E:firstblock}
			\int_{\mathbb{R}^d}|x-y|^{\varepsilon-s}\mu(dy)=(s-\varepsilon)\int_0^\infty m_x(r)r^{\varepsilon-s-1}dr
		\end{equation}
		for any $0<\varepsilon<s$, note that $\lim_{R\to\infty}m_x(R)R^{\varepsilon-s}=0$ since $\mu$ is finite and $\lim_{r\to 0}m_x(r)r^{\varepsilon-s}=0$ by \eqref{E:upperboundmu}. As in \eqref{E:large0}, 
		\begin{equation}\label{E:large}
			\lim_{\varepsilon\to 0}\varepsilon(s-\varepsilon)\int_1^\infty m_x(r)r^{\varepsilon-s-1}dr=0.
		\end{equation}
		Since $D_a^s\mu(x)$ exists, setting 
		$f(0):=sD_a^s\mu(x)$ and $f(u):=\frac{1}{|\log u|}\int_u^1m_x(r)r^{-s-1}dr$, $u>0$,
		defines a bounded function $f:[0,\infty)\to [0,\infty)$, continuous at zero. Fubini's theorem and \eqref{E:specific} give 
		\begin{multline}\label{E:secondblock}
			\int_0^1 k_\varepsilon(u)f(u)\:du=\int_0^1k_\varepsilon(u)\frac{1}{|\log u|}\int_u^1 m_x(r) r^{-s-1}dr\:du\\
			=\int_0^1r^{-s-1}m_x(r)\int_0^r \frac{k_\varepsilon(u)}{|\log u|}\:du\:dr=\varepsilon\int_0^1m_x(r)r^{\varepsilon-s-1}dr.
		\end{multline}
		Combining \eqref{E:firstblock}, \eqref{E:large}, \eqref{E:secondblock}, taking $\varepsilon\to 0$ and using \eqref{E:basicapprox},
		\[\lim_{\varepsilon\to 0}\int_{\mathbb{R}^d}|x-y|^{\varepsilon-s}\mu(dy)=\lim_{\varepsilon\to 0} \varepsilon(s-\varepsilon)\int_0^1 m_x(r)r^{\varepsilon-s-1}dr=s\lim_{\varepsilon\to 0} \int_0^1 k_\varepsilon(u)f(u)\:du=sD_a^s\mu(x).\]
	\end{proof}

	\begin{remark}  We do not claim that the existence of the limit on the left-hand side of \eqref{E:limitofpot} entails the existence of $D_{a}^{s}\mu(x)$.
	\end{remark}
	
	The observation that \eqref{E:upperboundmu} is a sufficient condition is new and convenient for applications to processes in transient regimes, for which local times do not exist. 
	
	\begin{example}\mbox{}\label{Ex:classical}  For isotropic $\beta$-stable processes $X$ on $\mathbb{R}^d$, $0<\beta\leq 2$, $\beta<d$, $\beta\neq 1$, it was shown in \cite[Lemma 2.3 (b)]{PerkinsTaylor1987} that there is an event $\Omega_1$ of probability one such that, with a constant $c>0$, 
		\begin{equation}\label{E:PerkinsTaylor}
			\limsup_{r\to 0}\frac{\mu_{0,T}^{X(\omega,\cdot)}(B(x,r))}{r^\beta|\log r|}\leq c, \qquad x\in \{X(\omega,t):\ t\in [0,T]\}
		\end{equation}
		for all $\omega\in \Omega_1$. For $d\geq 3$ and $\beta=2$ we obtain the corresponding result for the special case of Brownian motion, for which \eqref{E:PerkinsTaylor} had already been conjectured \cite[p. 201]{Taylor1974}; see \cite[p. 2]{Demboetal2000} for further comments. The bound \eqref{E:PerkinsTaylor} gives \eqref{E:upperboundmu} with $s=\beta$ and $\ell(r)=|\log r|$ for such $\omega$ and $x$. Together with \cite[Theorem 1]{FalconerXiao1995} and Lemma \ref{L:limitsofpots} it follows that
		\[\lim_{\varepsilon\to 0}\varepsilon\;U^{d-\beta+\varepsilon}\mu_{0,T}^{X(\omega,\cdot)}(x) = \beta \;D_{a}^{\beta}\mu_{0,T}^{X(\omega,\cdot)}(x)=\beta\:c_\beta\]
		for all $\omega$ from an event $\Omega_2$ of full probability and all $x\in\{X(\omega,t):\ t\in [0,T]\}\setminus N_\omega$, where $N_\omega$ is a $\mu_{0,T}^{X(\omega,\cdot)}$-null set. Here $c_\beta$ is the constant in Remark \ref{R:literature} (ii). The choices $d\geq 3$ and $\beta=2$ give the result in the special case of (transient) Brownian motion.
	\end{example}
	
	\begin{remark} The local asymptotics in \cite[Theorem 3]{CiesielskiTaylor1962} and \cite[Theorem 1]{Ray1963} respectively \cite[Theorem 4]{Taylor1967} hold on events of probability one depending on the chosen point of the image (roughly speaking, on $x$), so that they do not combine easily with \cite[Theorem 1]{FalconerXiao1995}.
	\end{remark}

	\section{Proofs of the main results}\label{S:proofs}
	
	Given $\gamma>0$, let $\psi_\gamma:(0,\infty)\to (0,\infty)$ be defined as
	\begin{equation}\label{E:psi}
		\psi_\gamma(s):=\mathbf{1}_{(0,1)}(s)s^{-\gamma}+\mathbf{1}_{[1,\infty)}(s)\log (e+s),\qquad s>0.
	\end{equation}
	
	\begin{prop} 
		\label{prop:basic-estimate}
		Suppose that the process $X=(X_t)_{t\in [0,T]}$ satisfies Assumption \ref{a1} (i), let $q'\geq \max(Hd,1)$ and $\gamma<\frac{d}{q'}$. Then for any $0 < u < v\leq T$ and any $n \geq 1$, we have
		\begin{multline}\label{E:desired}
			\int_{[u, v]^n} \Big(\int_{(\mathbb{R}^{d})^n}  \prod_{k=1}^n\psi_\gamma(\Vert \xi_k\Vert)^{q'}\Big\lvert \mathbb{E} \exp \Big( i \sum_{j = 1}^n \langle \xi_j, X_{t_j}\rangle \Big) \Big\rvert^{q'} d\xi\Big)^{\frac{1}{q'}} dt  \\
			\leq c_7^n n^{\left(H\theta  + \frac{Hd}{q'}\right)n}(v-u)^{\left(1-\frac{Hd}{q'}\right)n}[\log(e+(v-u))]^{2 n d}
		\end{multline} 
		with a constant $c_7 > 0$ depending only on $T$, $d$, $\gamma$, $c_0$, $H$ and $q'$.
	\end{prop} 
	
	\begin{remark}\mbox{}
		\begin{enumerate}
			\item[(i)]	For $Hd<1$ we can choose $q'=1$; this corresponds to the local time regime.
			\item[(ii)] Proposition \ref{prop:basic-estimate} and its proof are a substantial upgrade of the basic idea introduced in \cite[Lemma 3.2]{bouf2} and used in \cite[Proposition 3.2]{SSV}. New challenges are the singularity at zero contributed by the function $\psi_\gamma$ and the need for a careful analysis of logarithmic terms. It is exactly due to $\psi_\gamma$ and these logarithmic terms that we arrive at the natural exponents in Theorem \ref{main}.
		\end{enumerate}
	\end{remark}
	
	\begin{remark}\label{R:lHospital}
		The following proof uses that for any $a$ larger than a certain threshold $a(q')>0$, 
		\begin{equation}\label{E:elem}
			\int_a^\infty [\log(e+x)]^{2q'}x^{-4q'}dx \leq \frac{1}{4q'-2}[\log(e+a)]^{2q'}a^{-4q'+1},
		\end{equation}
		as can be seen by l'Hospital's rule. It also uses the fact that
		\begin{equation}\label{E:elem2}		
			b(q',H):=\int_{T^{-H}}^\infty [\log(e+x)]^{2q'}x^{-4q'}dx<\infty.
		\end{equation}
	\end{remark}

	\begin{proof}	
		Performing the change of variables $\eta_j \coloneqq \sum_{\ell=j}^n\xi_\ell$, $j=1,\ldots ,n$ with the convention that $\eta_{n+1}:=0$, we see that
		\begin{align}\label{E:ansatz}
			I &:= \int_{[u, v]^n} \Big(\int_{(\mathbb{R}^{d})^n} \prod_{k=1}^n \psi_\gamma(\Vert \xi_k\Vert)^{q'} \big\lvert \mathbb{E} \exp \Big( i \sum_{j = 1}^n \langle\xi_j, X_{t_j}\rangle \Big) \Big\rvert^{q'} d\xi\Big)^{\frac{1}{q'}} dt  \notag\\
			&= n!   \int_{u\leq t_1<t_2<\ldots < t_n\leq v} \Big(\int_{(\mathbb{R}^{d})^n}  \prod_{k=1}^n \psi_\gamma(\Vert \xi_k\Vert)^{q'}  \Big\lvert \mathbb{E} \exp \Big( i \sum_{j = 1}^n \langle\xi_j, X_{t_j}\rangle \Big) \Big\rvert^{q'} d\xi\Big)^{\frac{1}{q'}} dt  \notag\\
			&=n!   \int_{u\leq t_1<t_2<\ldots < t_n\leq v} \Big(\int_{(\mathbb{R}^{d})^n} \prod_{k=1}^n \psi_\gamma(\Vert \eta_k-\eta_{k+1}\Vert)^{q'} \times\notag\\
			&\qquad\qquad\qquad\qquad\qquad\qquad\times \Big\lvert \mathbb{E} \exp \Big( i \sum_{j = 1}^n \langle\eta_j, (X_{t_j}-X_{t_{j-1}})\rangle \Big) \Big\rvert^{q'} d\eta\Big)^{\frac{1}{q'}} dt.
		\end{align}

		To enable an efficient use of Assumption \ref{a1} (i), we decompose the domain of integration simultaneously in two ways, once according to the size of the differences $\eta_k-\eta_{k+1}\in\mathbb{R}^d$ and once according to the size of the single coordinates $\eta_k^\ell\in \mathbb{R}$ relative to the corresponding time differences $t_k-t_{k-1}$. Given $\eta_0\in\mathbb{R}^d$, let 
		\[D_0(\eta_0):=B_1(\eta_0)\quad \text{and}\quad D_1(\eta_0):=\mathbb{R}^d\setminus B_1(\eta_0).\]	
		For any fixed $t=(t_1,t_2,...,t_n)$ with $u\leq t_1<t_2<\ldots < t_n\leq v$ and any $k\in \{1,\dots, n\}$ let 
		\[J_{k,0}(t):=[-(t_k-t_{k-1})^{-H},(t_k-t_{k-1})^{-H}]\quad \text{and}\quad J_{k,1}(t):=\mathbb{R}\setminus J_{k,0}(t);\]
		here we agree to set $t_0:=u$. Given an arbitrary word $z_k=z_k^1\cdots\, z_k^d \in \{0,1\}^d$, we set
		\[Q_{z_k}(t):=J_{1,z_k^1}(t)\times \cdots \times J_{d,z_k^d}(t);\]
		this yields a partition $\mathbb{R}^d=\bigcup_{z_k\in \{0,1\}^d} Q_{z_k}(t)$ of $\mathbb{R}^d$. 
		The inner integral in \eqref{E:ansatz} equals
		\begin{multline}\label{E:inwords}
			\sum_{w\in \{0,1\}^n}\sum_{z_n\in \{0,1\}^d}\int_{D_{w_n}(\eta_{n+1})\cap Q_{z_n}(t)}\sum_{z_{n-1}\in \{0,1\}^d}\int_{D_{w_{n-1}}(\eta_{n})\cap Q_{z_{n-1}}(t)}\cdots  \\
			\sum_{z_1\in \{0,1\}^d}\int_{D_{w_1}(\eta_{2})\cap Q_{z_1}(t)}\prod_{k=1}^n\psi_\gamma(\Vert \eta_k-\eta_{k+1}\Vert)^{q'}\Big\lvert \mathbb{E} \exp \Big( i \sum_{j = 1}^n \langle\eta_j, (X_{t_j}-X_{t_{j-1}})\rangle \Big) \Big\rvert^{q'} d\eta_1\ \cdots\  d\eta_{n-1}d\eta_n,
		\end{multline}
		where in the first summation $w=w_1\cdots\, w_n$ runs through all words  of length $n$ over $\{0,1\}$, while in the other sums the $z_k$ run through the words of length $d$ over $\{0,1\}$. Given a fixed word $w\in \{0,1\}^n$ and an arbitrary choice of numbers $\kappa_{j,\ell}(w_j,z_j)\in \{0,4\}$, $j=1,\dots,n$, $\ell=1,\dots,d$, Assumption \ref{a1} (i) ensures that the summand in \eqref{E:inwords} indexed by $w$ does not exceed
		\begin{align}\label{E:iterate}
			&c_0^{nq'} n^{H\theta n q'}\sum_{z_n\in \{0,1\}^d}\int_{D_{w_n}(\eta_{n+1})\cap Q_{z_n}(t)}\sum_{z_{n-1}\in \{0,1\}^d}\int_{D_{w_{n-1}}(\eta_{n})\cap Q_{z_{n-1}}(t)}\cdots \sum_{z_1\in \{0,1\}^d}\int_{D_{w_1}(\eta_{2})\cap Q_{z_1}(t)} \notag\\
			&\hspace{75pt}\prod_{k=1}^n\psi_\gamma(\Vert \eta_k-\eta_{k+1}\Vert)^{q'}\prod_{j=1}^n\prod_{\ell=1}^d \lvert\eta_j^\ell\rvert^{-q'\kappa_{j,\ell}(w_j,z_j)}(t_j-t_{j-1})^{-q'H \kappa_{j,\ell}(w_j,z_j)}d\eta_1\ \cdots\ d\eta_{n-1}d\eta_n\notag\\
			&=c_0^{nq'} n^{H\theta n q'}\sum_{z_n\in \{0,1\}^d}\int_{D_{w_n}(\eta_{n+1})\cap Q_{z_n}(t)}\psi_\gamma(\Vert \eta_n-\eta_{n+1}\Vert)^{q'}\prod_{\ell_n=1}^d \lvert\eta_n^{\ell_n}\rvert^{-q'\kappa_{n,\ell_n}(w_n,z_n)}\notag\\
			&\hspace{40pt}\sum_{z_{n-1}\in \{0,1\}^d}\int_{D_{w_{n-1}}(\eta_{n})\cap Q_{z_{n-1}}(t)}\psi_\gamma(\Vert \eta_{n-1}-\eta_n\Vert)^{q'}\prod_{\ell_{n-1}=1}^d \lvert\eta_{n-1}^{\ell_{n-1}}\rvert^{-q'\kappa_{n-1,\ell_{n-1}}(w_{n-1},z_{n-1})}\ \cdots\notag\\
			&\hspace{80pt}\sum_{z_1\in \{0,1\}^d}\int_{D_{w_1}(\eta_{2})\cap Q_{z_1}(t)}\psi_\gamma(\Vert \eta_1-\eta_{2}\Vert)^{q'}\prod_{\ell_1=1}^d \lvert\eta_1^{\ell_1}\rvert^{-q'\kappa_{1,\ell_1}(w_1,z_1)}d\eta_1\ \cdots\ d\eta_{n-1}d\eta_n\ \times\\
			&\hspace{300pt}\times \prod_{j=1}^n (t_j-t_{j-1})^{-q'H \sum_{\ell=1}^d \kappa_{j,\ell}(w_j,z_j)}.\notag
		\end{align}
		
		We aim to estimate these iterated sums in an inductive manner, beginning with the innermost and progressing to the outermost. In the case $w_k=1$ the integral with respect to $\eta_k$ involves the factor $[\log(e+\Vert \eta_k-\eta_{k+1}\Vert)]^{q'}$; this is immediate from \eqref{E:psi}. Since the logarithm is increasing and satisfies $\log(e+s)\geq 1$ for $s\geq 0$, we have 
		\begin{multline}\label{E:logest}
			\log(e+\Vert \eta_k-\eta_{k+1}\Vert)\leq \log(e+\|\eta_k\|+\|\eta_{k+1}\|)\leq \log(e+\Vert \eta_k\Vert)+\log(e +\Vert\eta_{k+1}\Vert)\\
			\leq 2\log(e+\Vert \eta_k\Vert)\log(e +\Vert\eta_{k+1}\Vert)
		\end{multline}
		for all $\eta_k,\eta_{k+1}\in\mathbb{R}^d$. To use this bound in an inductive argument, we have to cover the factor involving $\eta_k$ by the current integration with respect to $\eta_k\in B_1(\eta_{k+1})$, while the factor involving $\eta_{k+1}$ gets passed on to the next iteration step. We claim that suitable choices of the numbers $\kappa_{k,\ell}(w_k,z_k)$ ensure that
		\begin{multline}\label{E:claim}
			\sum_{z_k\in \{0,1\}^d}\int_{D_{w_{k}}(\eta_{k+1})\cap Q_{z_k}(t)}\psi_\gamma(\Vert \eta_k-\eta_{k+1}\Vert)^{q'}[\log(e +\Vert\eta_{k}\Vert)]^{q'}
			\prod_{\ell=1}^d \lvert\eta_{k}^{\ell}\rvert^{-q'\kappa_{k,\ell}(w_{k},z_k)}d\eta_k\times\\
			\times(t_k-t_{k-1})^{-q'H \sum_{\ell=1}^d \kappa_{k,\ell}(w_k,z_k)}\\
			\leq c(d,H,q',\gamma)[\log(e+(v-u))]^{2dq'}(t_k-t_{k-1})^{-Hd}[\log(e +\Vert\eta_{k+1}\Vert)]^{q'}
		\end{multline}
		for any $k=1,...,n$, where 
		\[c(d,H,q',\gamma):=\max\Big(\frac{\omega_d d}{d-\gamma q'} 2^{q'}T^{Hd}[\log(e +1)]^{q'},2^{d+q'}d^{q'}c(q',H)\Big)\]
		with $c(q',H):=\max(2,b(q',H)a(q')^{4q'-1})$; here $a(q')$ and $b(q',H)$ are as in Remark \ref{R:lHospital}.
		
		If so, then a repeated application of \eqref{E:claim} to \eqref{E:iterate}, beginning with the innermost sum and progressing to the outermost, shows that \eqref{E:iterate} is bounded by 
		\[c(d,H,q',\gamma)^n c_0^{q' n} n^{H\theta n q'}[\log(e+(v-u))]^{2ndq'}\prod_{k=1}^n(t_k-t_{k-1})^{-Hd};\]
		note that $1\leq \log(e +\Vert\eta_{1}\Vert)$ and recall that $\eta_{k+1}=0$. Clearly \eqref{E:inwords} is then bounded by $2^n$ times this quantity.
		As a consequence,
		\[I\leq 2^nc(d,H,q',\gamma)^{\frac{n}{q'}}c_0^n n! n^{H\theta n}[\log(e+(v-u))]^{2nd} \int_{u\leq t_1<t_2<\ldots < t_n\leq v} \prod_{m=1}^n(t_m-t_{m-1})^{-\frac{Hd}{q'}}dt.\]
		A successive evaluation of Euler integrals (Beta functions), together with obvious cancellations, gives
		\[\int_{u\leq t_1<t_2<\ldots < t_n\leq v} \prod_{m=1}^n(t_m-t_{m-1})^{-\frac{Hd}{q'}}dt=\frac{\Gamma(1-\frac{Hd}{q'})^n}{n(1-\frac{Hd}{q'})\Gamma(n(1-\frac{Hd}{q'}))}(v-u)^{n(1-\frac{Hd}{q'})},\]
		and using the fact that, by Stirling's approximation,
		\[\frac{n!}{\Gamma(n(1-\frac{Hd}{q'}))}\leq c(d,H,q')^n n^{n\frac{Hd}{q'}}\]
		for large $n$ and with a suitable constant $c(d,H,q')>0$, we arrive at \eqref{E:desired}. The same bound for smaller $n$ can be guaranteed by readjusting the constant.
		
		It remains to show \eqref{E:claim}. Suppose first that $w_k=0$. Then $D_{0}(\eta_{k+1})=B_1(\eta_{k+1})$, and estimating similarly as in \eqref{E:logest}, we observe that
		\[\log(e+\|\eta_k\|)\leq 2\log(e+\|\eta_k-\eta_{k+1}\|)\log(e+\|\eta_{k+1}\|)\leq 2\log(e+1)\log(e+\|\eta_{k+1}\|)\]
		for all $\eta_k\in D_0(\eta_{k+1})$. Choosing $\kappa_{k,\ell}(0,z_k):=0$ for all $z_k$ and all $\ell$, the left-hand side of \eqref{E:claim} becomes
		\begin{align}
			\sum_{z_k\in \{0,1\}^d}\int_{D_0(\eta_{k+1})\cap Q_{z_k}(t)}&\Vert \eta_k-\eta_{k+1}\Vert^{-\gamma q'}[\log(e +\Vert\eta_{k}\Vert)]^{q'}d\eta_k \notag\\
			&\leq \omega_d d\: 2^{q'}[\log(e +1)]^{q'}[\log(e +\Vert\eta_{k}\Vert)]^{q'}\int_0^1 r^{d-\gamma q'-1}dr\notag\\
			&=\frac{\omega_d d}{d-\gamma q'}\:2^{q'}[\log(e +1)]^{q'}[\log(e +\Vert\eta_{k}\Vert)]^{q'},\notag
		\end{align}
		and using $1\leq T^{Hd}(t_k-t_{k-1})^{-Hd}$ and $1\leq [\log(e+(v-u))]^{2dq'}$, we arrive at \eqref{E:claim} for $w_k=0$. Now suppose that that $w_k=1$. The estimate 
		\begin{equation}\label{E:logclaim}
			\log(e+\Vert \eta_k\Vert)\leq d\prod_{\ell=1}^d\log(e+|\eta_k^\ell|)
		\end{equation}
		follows similarly as \eqref{E:logest}.
		We choose $\kappa_{k,\ell}(1,z_k) =4z_k^\ell$ for all $z_k$ and all $\ell$. For $z_k^\ell=0$ then obviously 
		\begin{multline}\label{E:verysmallclaim}
			\int_{J_{k,0}} [\log(e+|\eta_k^\ell|)]^{2q'}|\eta_k^\ell|^{-q'\kappa_{k,\ell}(1,z_k)}d\eta_k^\ell (t_k-t_{k-1})^{-Hq'\kappa_{k,\ell}(1,z_k)}\\
			\leq  c(q',H)[\log(e+(t_k-t_{k-1}))]^{2q'}(t_k-t_{k-1})^{-H}.
		\end{multline}
		For $z_k^\ell=1$ we have
		\begin{align}\label{E:smallclaim}
			\int_{J_{k,1}} [\log(e+|\eta_k^\ell|)]^{2q'}&|\eta_k^\ell|^{-q'\kappa_{k,\ell}(1,z_k)}d\eta_k^\ell (t_k-t_{k-1})^{-Hq'\kappa_{k,\ell}(1,z_k)}\\
			&\leq  c(q',H)[\log(e+(t_k-t_{k-1}))]^{2q'}(t_k-t_{k-1})^{-H(-4q'+1)}(t_k-t_{k-1})^{-4Hq'} \notag\\
			&= c(q',H)[\log(e+(t_k-t_{k-1}))]^{2q'}(t_k-t_{k-1})^{-H}.\notag
		\end{align}
		This can be seen using Remark \ref{R:lHospital}: In the case that $(t_k-t_{k-1})^{-H}>a(q')$,
		an application of \eqref{E:elem} with $a=(t_k-t_{k-1})^{-H}$ and the fact that $q'\geq 1$ give the inequality in \eqref{E:smallclaim}.
		To see \eqref{E:smallclaim} in the remaining case that $T^{-H}\leq (t_k-t_{k-1})^{-H}\leq a(q')$ we can use \eqref{E:elem2},
		together with the fact that $(t_k-t_{k-1})^{-4Hq'}\leq a(q')^{4q'-1}(t_k-t_{k-1})^{-H}$. Combining \eqref{E:verysmallclaim} and \eqref{E:smallclaim} with \eqref{E:logest} and \eqref{E:logclaim}, we obtain the estimate
		\begin{multline}
			\sum_{z_k\in \{0,1\}^d}\int_{D_1(\eta_{k+1})\cap Q_{z_k}(t)}[\log(e+\Vert \eta_k-\eta_{k+1}\Vert)]^{q'}[\log(e +\Vert\eta_{k}\Vert)]^{q'}
			\prod_{\ell=1}^d \lvert\eta_{k}^{\ell}\rvert^{-q'\kappa_{k,\ell}(1,z_k)}d\eta_k\times\\
			\hspace{150pt}\times(t_k-t_{k-1})^{-q'H \sum_{\ell=1}^d \kappa_{k,\ell}(1,z_k)}\\
			\leq 2^{d+q'}c(q',H)^d d^{q'}[\log(e+(t_k-t_{k-1}))]^{2dq'}(t_k-t_{k-1})^{-Hd}[\log(e+\|\eta_{k+1}\|)]^{q'},
		\end{multline}
		which confirms \eqref{E:claim} also for $w_k=1$.
	\end{proof}

	\begin{cor}
		\label{cor:basic-estimate}
		Suppose that $X=(X_t)_{t\in[0,T]}$ satisfies Assumption \ref{a1} (i). Let $0\leq \alpha < d$ and let $0\leq \beta \leq 1$ be such that 
		\begin{equation}
			\label{eq:key-restriction}
			\beta - \alpha < \min\left(1,\frac{1-Hd}{2H},\frac{1-Hd}{H} \right).
		\end{equation}
		Then there exists a constant $c_8>0$ depending only on $T$, $d$, $H$, $\alpha$ and $\beta$, such that for any times $0 < u < v\leq T$ we have
		\begin{multline}
			\label{eq:bound_fourier}   \int_{[u, v]^n} \int_{(\mathbb{R}^{d})^n} \prod_{k =1}^n \|\xi_k\|^{\beta-\alpha} \Big\lvert \mathbb{E} \exp \Big( i \sum_{j = 1}^n \langle \xi_j, X_{t_j}\rangle \Big) \Big\rvert d\xi dt  \\
			\leq c_8^n n^{H(\theta+d-\alpha+\beta)n}(v-u)^{(1-H(d-\alpha+\beta))n}[\log(e+(v-u))]^{2nd\zeta}
		\end{multline}
		where $\zeta=1$ if $\beta<\alpha$ while $\zeta=0$ if $\beta\geq \alpha$.
	\end{cor}

	\begin{remark}\label{rem:parameters}
		The condition \eqref{eq:key-restriction} on $\beta$ implies that
		\begin{equation}\label{E:intermedcond}		
			1 - \frac{\alpha-\beta}{d} <\frac{1}{Hd}.
		\end{equation}		
		Since $\beta\geq 0$, \eqref{E:intermedcond} requires  $\alpha > d-\frac{1}{H}$. This is possible even if $Hd\geq 1$ (provided that $\alpha$ is sufficiently large), in which case the right-hand side of \eqref{eq:key-restriction} equals $\frac{1-Hd}{H}$. 
		If $\beta\geq \alpha$, then \eqref{E:intermedcond} implies $Hd<1$ (the local time regime) and, in particular, the right-hand side of \eqref{eq:key-restriction} equals $\min\left(1,\frac{1-Hd}{2H} \right)$, which is the known standard bound in the case $\alpha=0$. We are mainly interested in the case $\beta<\alpha$ (which excludes $\alpha=0$). In this case there is a unique number $q'>1$ such that 
		\begin{equation}\label{E:qprime}
			1 - \frac{\alpha-\beta}{d} = \frac{1}{q'}.  
		\end{equation}
	\end{remark}

	\begin{proof}[Proof of Corollary \ref{cor:basic-estimate}] For the case $\beta\geq \alpha$ (where $\zeta=0$) the result follows by \cite[Proposition 3.2]{SSV}, applied with $\beta - \alpha$ in place of $\eta$. We therefore assume that $\beta < \alpha$. H\"older's inequality with $q'$ as in \eqref{E:qprime} and $\frac{1}{q}+\frac{1}{q'} = 1$, applied to the inner integral, gives
		\begin{multline}
			\int_{[u, v]^n} \int_{(\mathbb{R}^{d})^n} \prod_{k =1}^n \|\xi_k\|^{\beta-\alpha} \Big\lvert \mathbb{E} \exp \Big( i \sum_{j = 1}^n \langle \xi_j, X_{t_j}\rangle \Big) \Big\rvert d\xi dt \leq \Big(\int_{(\mathbb{R}^{d})^n} \prod_{k =1}^n \frac{\|\xi_k\|^{{q(\beta-\alpha)}}}{\psi_\gamma(\Vert \xi_k\Vert)^q}   d\xi\Big)^{\frac{1}{q}} \times\notag\\
			\times \int_{[u, v]^n} \Big(\int_{(\mathbb{R}^{d})^n} \prod_{k=1}^n \psi_\gamma(\Vert \xi_k\Vert)^{q'}\Big\lvert \mathbb{E} \exp \Big( i \sum_{j = 1}^n \langle \xi_j, X_{t_j}\rangle \Big) \Big\rvert^{q'} d\xi\Big)^{\frac{1}{q'}} dt.
		\end{multline}
		The first factor on the right-hand side equals 
		\[\Big(\prod_{k=1}^n \int_{\mathbb{R}^d} \frac{\|\xi_k\|^{{q(\beta-\alpha)}}}{\psi_\gamma(\Vert \xi_k\Vert)^q}   d\xi_k\Big)^{1/q},\] 
		and since $q(\beta-\alpha) =-d$ and $q>1$, the individual integrals all equal
		\[\int_{B_1}\|\xi_1\|^{\gamma q-d}d\xi_1+\int_{\mathbb{R}^d\setminus B_1}\frac{\|\xi_1\|^{-d}}{[\log(e+\|\xi_1\|)]^q}d\xi_1\leq c(d,q),\]
		where $B_1$ denotes the unit ball in $\mathbb{R}^d$ and $c(d,q)$ is a constant depending only on $d$ and $q$. For the second factor on the right-hand side, we can apply Proposition \ref{prop:basic-estimate} since $q'>Hd$ by \eqref{E:intermedcond} and \eqref{E:qprime}.
		
	\end{proof}

	\begin{prop}
		\label{prop:moment-bounds}
		Suppose that $X=(X_t)_{t\in[0,T]}$ satisfies Assumption \ref{a1} (i). Let $0\leq \alpha < d$ and $0\leq \beta \leq 1$ be such that 
		\eqref{eq:key-restriction} holds and let $n$ be an even integer. Then for every $0 \leq u < v\leq T$ and $x,h \in \rd$ we have
		\begin{equation}\label{eq:mom3}
			\E \lvert L^{\alpha,X}(x,[u,v])\rvert^n \leq c_9^n n^{H\left(\theta+d-\alpha\right)n}(v-u)^{\left(1-H\left(d-\alpha\right)\right)n}[\log(e+(v-u))]^{2nd\zeta}
		\end{equation}
		and
		\begin{multline}\label{eq:mom4}
			\E [L^{\alpha,X}(x+h,[u,v])-L^{\alpha,X}(x,[u,v])]^{n}\\
			\leq \|h\|^{\beta n}   c_9^n n^{H\left(\theta+d-\alpha+\beta\right)n}(v-u)^{\left(1-H\left(d-\alpha+\beta\right)\right)n}[\log(e+(v-u))]^{2nd\zeta}
		\end{multline}
		with a constant $c_9>0$ depending only on $T$, $d$, $H$, $\alpha$ and $\beta$ and where $\zeta=1$ if $\beta<\alpha$ while $\zeta=0$ if $\beta\geq \alpha$.
	\end{prop}
	
	\begin{remark}
		The case $\alpha=0$ in Proposition \ref{prop:moment-bounds} was already part of \cite[Theorem 5.7]{HTV2025}.
	\end{remark}	
	
	\begin{proof}
		For any $0<\beta\leq 1$ we can proceed as in \cite[proof of Theorem 5.7]{HTV2025} to see that 
		\begin{multline}\label{E:intermed}
			\mathbb{E}[L^{\alpha,X}(x+h,[u,v])-L^{\alpha,X}(x,[u,v])]^n\\
			\leq \|h\|^{\beta n}\:
			\int_{[u,v]^n}\int_{(\mathbb{R}^d)^n}\prod_{j=1}^n \|\xi_j\|^{(\beta-\alpha)}\Big\lvert \mathbb{E} \exp \Big( i \sum_{j = 1}^n \langle \xi_j, X_{t_j}\rangle \Big) \Big\rvert\:d\xi\:dt.
		\end{multline}
		In this case \eqref{eq:mom4} is then immediate from Corollary \ref{cor:basic-estimate}. Inequality \eqref{eq:mom3} follows similarly using  Corollary \ref{cor:basic-estimate} with $\beta=0$. 
	\end{proof}
	
	The next result is a slight modification of Proposition \ref{prop:moment-bounds}, where one shifts the process in space by the random variable $X_\tau$ indexed by a fixed time point $\tau$. It differs from related results in \cite{knsv}, note that here we do not assume $X$ to have stationary increments.

	\begin{corollary}
		\label{thm:moment} Suppose that $X=(X_t)_{t\in[0,T]}$ satisfies Assumption \ref{a1} (i). Let $0\leq \alpha < d$ and $0\leq \beta \leq 1$ be such that \eqref{eq:key-restriction} holds and let $n$ be an even integer. Then for any $0 \leq u < v\leq T$, any $x,y \in \rd$ and any $\tau\in[u,v]$ we have
		\begin{equation}\label{eq:mom3alt}
			\E \lvert L^{\alpha,X}(x+X_\tau,[u,v])\rvert^n \leq c_9^n n^{H\left(\theta+d-\alpha\right)n}(v-u)^{\left(1-H\left(d-\alpha\right)\right)n}[\log(e+(v-u))]^{2nd \zeta}
		\end{equation}
		and
		\begin{multline}\label{eq:mom4alt}
			\E [L^{\alpha,X}(x+h+X_\tau,[u,v])-L^{\alpha,X}(x+X_\tau,[u,v])]^{n}  \\
			\leq \|h\|^{\beta n}   c_9^n n^{H\left(\theta+d-\alpha+\beta\right)n}(v-u)^{\left(1-H\left(d-\alpha+\beta\right)\right)n}[\log(e+(v-u))]^{2nd\zeta},
		\end{multline}
		where $c_9$ and $\zeta$ are as in Proposition \ref{prop:moment-bounds}.
	\end{corollary}
	\begin{proof} Since the addition of $X_\tau$ leaves the increments $X_{t_j}-X_{t_{j-1}}$ in the last line of \eqref{E:ansatz} unchanged,  Proposition \ref{prop:basic-estimate}, Corollary \ref{cor:basic-estimate} and Proposition \ref{prop:moment-bounds} remain valid with $X+X_\tau$ in place of $X$.
	\end{proof}

	Combining Proposition \ref{prop:moment-bounds} and the Kolmogorov criterion \cite[Theorem 1.4.4]{Kunita1990},
	we obtain the joint H\"older continuity of $L^{\alpha,X}$. This result may be of independent interest, it generalizes the continuity results in
	\cite[Theorem 5.7]{HTV2025} and \cite[Corollary 3.4]{SSV} beyond the Gaussian case respectively the local time regime.
	\begin{theorem}\label{thm:holder} 
		Suppose that $X=(X_t)_{t\in[0,T]}$ satisfies Assumption \ref{a1} (i). Let $0\leq \alpha < d$ and $0\leq \beta \leq 1$ be such that 
		\eqref{eq:key-restriction} holds. The random field $(x,t)\mapsto L^{\alpha,X}(x,[0,t])$ (has a modification which) is $\mathbb{P}$-a.s. jointly H\"older continuous in $(x,t)\in \mathbb{R}^d\times [0,T]$. 
		
		More precisely, for any cube $Q\subset \mathbb{R}^d$, any $0<\gamma_1<\beta$ and any $0<\gamma_2<1-H(d-\alpha)$  there exists a random variable $A$ such that $\mathbb{P}$-a.s. we have 
		\[|L^{\alpha,X}(x+h,[u,v])-L^{\alpha,X}(x,[u,v])|\leq A\|h\|^{\gamma_1}(v-u)^{\gamma_2}\]
		for all $x,x+h\in Q$ and all $0\leq u<v\leq T$.
	\end{theorem}

	By Chebyshev's inequality these moment bounds translate  into the following tail estimates. We omit the standard proof.
	
	\begin{corollary}\label{cor:tail1}
		Suppose that $X=(X_t)_{t\in[0,T]}$ satisfies Assumption \ref{a1} (i). Let $0\leq \alpha < d$ and $0\leq \beta \leq 1$ be such that \eqref{eq:key-restriction} holds, and let $\zeta$ be as in  Proposition \ref{prop:moment-bounds}.
		\begin{enumerate}	
			\item[(i)]  There are positive constants $c_{10}$ and $c_{11}$ such that for any $x \in \rd$, any subinterval $I\subset [0,T]$ and any $u>0$ we have
			\begin{equation}
				\label{eq:tail-L0}
				\mathbb{P}\Big(L^{\alpha,X}(x,I)\geq \lvert I\rvert ^{1-H(d-\alpha)}[\log(e+|I|)]^{2d\zeta} u^{H(\theta+d-\alpha)}\Big) \leq c_{10}\exp(-c_{11}u),
			\end{equation}
			and this remains true if $x$ is replaced by $x+X_\tau$ with $\tau\in I$.
			
			\item[(ii)] There are positive constants $c_{10}$ and $c_{11}$ such that for any $x,y \in \rd$, any subinterval $I\subset [0,T]$ and any $u>0$ we have
			\begin{multline}
				\label{eq:tail-inc0}
				\mathbb{P}\Big(\lvert L^{\alpha,X}(x,I)-L^{\alpha,X}(y,I)\rvert \geq  \lvert I\rvert ^{1-H(d-\alpha+\beta)}[\log(e+|I|)]^{2d\zeta} \lvert x-y\rvert ^{\beta} u^{H(\theta+d-\alpha+\beta)}\Big)\\
				\leq c_{10}\exp(-c_{11}u),
			\end{multline}
			and this remains true if $x$ and $y$ are replaced by $x+X_\tau$ respectively $y+X_\tau$ with $\tau\in I$.
		\end{enumerate}
	\end{corollary}
	
	Theorem \ref{main} now follows using a standard chaining method, see e.g. \cite[Theorem 1.2]{Xiao1997} or \cite[Theorem 1.4]{knsv}. However, in comparison to earlier related works the exponents behave quite differently and the logarithmic terms have to be taken into account. We therefore present the main arguments.
	
	\begin{proof}[Proof of Theorem \ref{main}]
		For $Hd<1$ and $\alpha=0$ the claim of Theorem \ref{main} is a restatement of \cite[Theorem 2.1]{SSV}, in line with the $\zeta=0$ case of the preceding statements. 
		We provide a proof of Theorem \ref{main} for the case not covered by \cite[Theorem 2.1]{SSV}, where $H\in(0,1)$ and
		\[\max\big(0,d-\frac{1}{H}\big)<\alpha<d.\]
		Clearly $L^{\alpha,X}(x,[t-r,t+r])=L^{\alpha,X}(x,[t-r,t])+L^{\alpha,X}(x,[t,t+r])$. We consider $t\in (0,T)$ as fixed and prove \eqref{eq:main-fixed-s} with the second summand in place of $L^{\alpha,X}(x,[t-r,t+r])$; its proof for the first summand follows by obvious modifications. By Assumption \ref{a1} (ii) and the Kolmogorov-Chentsov theorem we may restrict our attention to an event $\Omega_0$ of probability one and assume that $X$ is continuous. As a consequence, $X([t,t+r])$ is compact, hence closed, so that $\supp\mu_{t,t+r}^X=X([t,t+r])$. The first maximum principle, cf. \cite[Section I.3, Theorem 1.5]{Landkof1972}, ensures that
		\[\sup_{x\in\mathbb{R}^d}L^{\alpha,X}(x,[t,t+r])\leq 2^{d-\alpha}\sup_{x\in X([t,t+r])}L^{\alpha,X}(x,[t,t+r]).\]
		Now let 
		\[g(r):= r^{1-H (d-\alpha)}[\log (e+r)]^2 [\log \log r^{-1}]^{H(\theta+d-\alpha)},\qquad 0<r<1.\] 
		Since 
		\[\lim_{r\to 0} \frac{g(r)}{r^{1-H (d-\alpha)} (\log \log r^{-1})^{H(\theta+d-\alpha)}} = 1,\]	
		$g$ satisfies a doubling property and $r\mapsto \sup_{x\in X([t,t+r])}L^{\alpha,X}(x,[t,t+r])$ is increasing in $r$, 
		it suffices to prove that there is a constant $c_{12}>0$ such that 
		\begin{equation}\label{E:target}
			\limsup_{n\to\infty} \sup_{x\in X(C_n)}\frac{L^{\alpha,X}(x,C_n)}{g(2^{-n})} <c_{12}\quad \text{$\mathbb{P}$-a.s.},
		\end{equation}
		where $C_n := [t,t+2^{-n}]$. 
		
		By \cite[Corollary 2.11]{Nummi-Viitasaari} there is an event $\Omega_1\subset \Omega_0$ of probability one and a constant $c_{13}>0$ such that for all $\omega\in \Omega_1$ there is some  $n_1(\omega)\in\mathbb{N}$ such that
		\[\sup_{s\in C_n}\| X_s(\omega)-X_t(\omega)\| \leq c_{13} 2^{-nH} n^{\iota},\qquad n\geq n_1(\omega).\]
		Setting 
		\begin{equation}\label{E:lambdan}
			\lambda_n := 2^{-nH}(\log \log 2^n)^{-H}
		\end{equation} 
		and 
		\[G_n :=\{x \in \lambda_n\Z^d : \| x\| \leq c_{13}2^{-nH}n^{\iota}\},\]
		we find that 
		\begin{equation}\label{E:volume}		
			\# G_n \leq c_{14} (\log n)^{H d}n^{\iota d}. 
		\end{equation}		
		
		Let $c_{15}>0$ be large enough to have $c_{16}:=c_{11}c_{15}^{1/(H(\theta+d-\alpha))}>\iota d+2$. Using \eqref{eq:tail-L0} with $\zeta=1$ and with $u_n^{H(\theta+d-\alpha)}=c_{15}(\log(n\log 2))^{H(\theta+d-\alpha)}$ in place of $u$, we obtain 
		\[\mathbb{P}\Big(L^{\alpha,X}(x+X_t,C_n) \geq c_{15} g(2^{-n})\Big)\leq c_{10}n^{-c_{16}}.\]
		
		Together with \eqref{E:volume} this gives 
		\[\mathbb{P}\Big(\max_{x\in G_n} L^{\alpha,X}(x+X_t,C_n) \geq c_{15} g(2^{-n})\Big)\leq c_{10} (\log n)^{H d} n^{\iota d -c_{16}},\]
		which is summable. By Borel-Cantelli we can find an event $\Omega_2\subset \Omega_1$ of probability one and $n_2(\omega)\geq n_1(\omega)$, $\omega \in \Omega_2$, such that 
		\begin{equation}
			\label{eq:max-L-bounded}
			\max_{x\in G_n} L^{\alpha,X}(x+X_t,C_n)(\omega) \leq c_{15} g(2^{-n})\qquad \text{for all $n\geq n_2(\omega)$}.
		\end{equation}
		
		Given integers $n,k\geq 1$ and a point $x\in G_n$, we set
		\[F(n,k,x) := \{ y = x + \lambda_n \sum_{j=1}^k \varepsilon_j2^{-j} \;:\; \varepsilon_j \in \{0,1\}^d, 1\leq j \leq k\}.\]
		Similarly as in \cite{Xiao1997} we call a pair of two points $y_1,y_2 \in F(n,k,x)$ linked if $y_2-y_1 = \lambda_n \varepsilon 2^{-k}$ for $\varepsilon \in\{0,1\}^d$. Let $0<\beta<\alpha$ satisfy \eqref{eq:key-restriction}, choose $\delta >0$ such that 
		\begin{equation}\label{E:deltacond}		
			\delta [H\theta  + H(d-\alpha+\beta)]<\beta
		\end{equation}		
		and set
		\begin{multline}\label{E:Bns}
			B_n := \bigcup_{x\in G_n}\bigcup_{k=1}^\infty \bigcup_{y_1,y_2} \Big\lbrace \lvert L^{\alpha,X}(y_1+X_s,C_n)-L^{\alpha,X}(y_2+X_s,C_n)\rvert \\
			\geq 2^{-n(1-H(d-\alpha+\beta))}[\log 2^n]^2\|y_1-y_2\|^{\beta}(c_{17}2^{\delta k}\log n)^{H\theta  + H(d-\alpha+\beta)}\Big\rbrace ,
		\end{multline}
		where $\bigcup_{y_1,y_2}$ is the union over all linked pairs $y_1,y_2\in F(n,k,x)$ and $c_{17}>0$. Now \eqref{eq:tail-inc0} with $u = c_{17}2^{\delta k}\log n$ gives
		\[\mathbb{P}(B_n) \leq c_{14}(\log n)^{H d}n^{\iota d}\sum_{k=1}^\infty 4^{kd}\exp\left(-c_{18}2^{\delta k}\log n\right),\]
		where $c_{18}:=c_{11}c_{17}$. Here we have used \eqref{E:volume} and that for each $k$ there are no more than $4^{kd}$ linked pairs $y_1,y_2$. Choosing $c_{17}$ large enough, we can ensure that $n^{-c_{18}2^{\delta k}} \leq n^{-pk}$, $k\geq 1$, with a fixed number $p>\iota d+2$. 
		For all $n$ with $4^d < n^p$ the summation over $k$ then gives  
		\[\sum_{k=1}^\infty 4^{kd}\exp\left(-c_{18}2^{\delta k}\log n\right) \leq \sum_{k=1}^\infty 4^{kd}n^{-pk} \leq \frac{4^{d}}{n^p-4^d}\]
		Now the product of the right-hand side and $(\log n)^{H d}n^{\iota d}$ is summable with respect to $n$, 
		so that
		\[\sum_{n=2}^\infty \mathbb{P}(B_n) \leq c_{14}
		\sum_{n=2}^\infty(\log n)^{H d}n^{\iota d}\sum_{k=1}^\infty 4^{kd}\exp\left(-c_{18}2^{\delta k}\log n\right) < \infty\]
		and by Borel-Cantelli the event $B_n$ occurs only finitely many times.

		Let $n$ be fixed and assume that $y \in \rd$ satisfies $\| y\|  \leq 2c_12^{-n H} n^{\iota}$. Then we may represent $y$ as $y = \lim_{k\to\infty}  y_k$ with
		\begin{equation}\label{E:yks}
			y_k := x + \lambda_n \sum_{j=1}^k \varepsilon_j 2^{-j},\qquad k\geq 1,
		\end{equation}
		where $\varepsilon_j \in \{0,1\}^d$ and $x \in G_n$; we complement this by setting $y_0:= x$. By the preceding we can find an event $\Omega_3\subset \Omega_2$ of probability one and $n_3(\omega)\geq n_2(\omega)$, $\omega \in \Omega_3$, such that for all $\omega\in \Omega_3$ and all $n\geq n_3(\omega)$ we have 
		\begin{equation*}
			\begin{split}
				\lvert L^{\alpha,X}(x+X_t,C_n)&(\omega)-L^{\alpha,X}(y+X_t,C_n)(\omega)\rvert  \notag\\ 
				&\leq \sum_{k=1}^\infty \lvert L^{\alpha,X}(y_k+X_t,C_n)(\omega)-L^{\alpha,X}(y_{k-1}+X_t,C_n)(\omega)\rvert  \\
				&\leq \sum_{k=1}^\infty 2^{-n(1-H(d-\alpha+\beta))}[\log 2^n]^2\|y_k-y_{k-1}\|^{\beta}(c_{17}2^{\delta k}\log n)^{H\theta  + H(d-\alpha+\beta)} \\
				&\leq c_{19} 2^{-n(1-H(d-\alpha+\beta))}[\log 2^n]^2\lambda_n^{\beta}\sum_{k=1}^\infty 2^{-k\beta}(c_{17}2^{\delta k}\log n)^{H\theta  + H(d-\alpha+\beta)} \\
				& \leq c_{20} 2^{-n(1-H(d-\alpha))}[\log 2^n]^2 (\log \log 2^n)^{-\beta H}\sum_{k=1}^\infty 2^{-k\beta}(c_{17}2^{\delta k}\log n)^{H\theta  + H(d-\alpha+\beta)} \\
				& \leq c_{21} 2^{-n(1-H(d-\alpha))} [\log 2^n]^2(\log \log 2^n)^{H\theta  + H(d-\alpha)} \sum_{k=1}^\infty 2^{(\delta [H\theta  + H(d-\alpha+\beta)]-\beta)k} \\
				& \leq c_{22}g(2^{-n});
			\end{split}
		\end{equation*}
		here we have used \eqref{E:Bns}, \eqref{E:yks} and \eqref{E:lambdan}. The last inequality follows from 
		\eqref{E:deltacond}. Combining this with \eqref{eq:max-L-bounded} then yields
		\[\sup_{\|x\|\leq 2c_12^{-nH} n^{\iota}} L^{\alpha,X}(x+X_s,C_n)(\omega) \leq c_{12} g(2^{-n}),\qquad \omega\in \Omega_3,\ n\geq n_3(\omega),\]
		or, in other words,
		\[\sup_{\|x-X_s\|\leq 2c_12^{-nH} n^{\iota}} L^{\alpha,X}(x,C_n)(\omega) \leq c_{12}g(2^{-n}),\qquad \omega\in \Omega_3,\ n\geq n_3(\omega).\]
		This gives \eqref{E:target} and therefore \eqref{eq:main-fixed-s}. 
		
		The limit relation \eqref{eq:main-sup-s} follows similarly by straightforward modifications of the arguments, see for instance \cite[Theorem 1.4]{knsv} or \cite[Theorem 1.2]{Xiao1997}. We leave the details to the reader.
	\end{proof}
	
	We prove Theorem  \ref{cm2}.
	
	\begin{proof}[Proof of Theorem \ref{cm2}] As before, we may assume that $\alpha>0$. For any sufficiently small $r>0$ we have $(t-r,t+r)\subset [0,T]$ and, using \eqref{eq:main-fixed-s},
\begin{align}
2r\:\Big(\sup_{s\in (t-r,t+r)}\|X_t-X_s\|\Big)^{\alpha-d}&\leq \int_{t-r}^{t+r}\|X_t-X_s\|^{\alpha-d}ds\notag\\
&=\frac{d\omega_d}{\alpha}L^{\alpha,X}(X_t,[t-r,t+r])\notag\\
&\leq \frac{d\omega_d}{\alpha}c_2(t)\:r^{1-H(d-\alpha)}(\log\log r^{-1})^{H(\theta+d-\alpha)}.\notag
\end{align}	
Rearranging, we obtain \eqref{E:output}. Using (\ref{eq:main-sup-s}) instead of \eqref{eq:main-fixed-s} gives \eqref{E:output2}.
	\end{proof}

	\noindent
	\textbf{Disclaimer.}
	Co-funded by the European Union. Views and opinions expressed are however those of the authors only and do not necessarily reflect those of
	the European Union or the European Education and Culture Executive Agency (EACEA). Neither the European Union nor EACEA can be held responsible for them.

%	\bibliographystyle{abbrv}
%	\bibliography{lit.bib}

\end{document}